\documentclass[11pt,a4paper]{amsart}
\usepackage[OT1]{fontenc}
\usepackage[english]{babel}
\usepackage{amsmath}
\usepackage{amsfonts}
\usepackage{amsthm}

\usepackage[bookmarks,pdfpagelabels,linkbordercolor={0.1 0.8 0.4}]{hyperref}

\def\RR{\mathbb{R}}

\def\h{ \mathcal{H} }
\def\l{ \mathcal{L} }
\def\b{ \mathcal{B} }

\def\t{ \mathcal{T} }
\def\s{ \mathcal{S} }
\def\e{ \mathcal{E} }
\def\p{ \mathfrak{P} }

\def\k{ \mathcal{K} }

\def\c{ C_{\s,\t} }
\def\cl{ \bar{C}_{\s,\t} }

\def\noi{\noindent}

\def\g1{ \mathfrak{g}_1  }

\def\vp{  \varphi }
\def\la{ \lambda }

\newcommand{\PI}[2]{\left\langle #1 , #2 \right\rangle}
\newcommand{\ov}[1]{\overline{#1}}

\newtheorem{teo}{Theorem}[section]
\newtheorem{prop}[teo]{Proposition}
\newtheorem{lem}[teo]{Lemma}
\newtheorem{coro}[teo]{Corollary}

\theoremstyle{definition}
\newtheorem{rem}[teo]{Remark}
\newtheorem{ejem}[teo]{Example}

\title{Proper subspaces and compatibility}
\author{E. Andruchow, E. Chiumiento and M.E. Di Iorio y Lucero}
\date{}

\begin{document}

\maketitle 

\begin{abstract}
Let $\e$ be a Banach space contained in a Hilbert space $\l$. Assume that the inclusion is continuous with dense range. 
Following  the terminology of Gohberg and Zambicki\v{\i}, we say that a bounded operator  on $\e$ is a proper operator
if it admits an adjoint with respect to the inner product of $\l$. By a proper subspace $\s$ we mean a closed subspace  of $\e$ which
is the range of a proper projection. If there exists a proper projection which is also self-adjoint with respect to the inner product of $\l$, then  $\s$ belongs to a well-known class of subspaces called compatible subspaces. We find equivalent conditions to describe proper subspaces. Then we prove a necessary and sufficient condition to ensure that a proper subspace is compatible. Each proper subspace $\s$ has a supplement $\t$ which is also a proper subspace.  We  give a characterization of the compatibility of both subspaces $\s$ and $\t$. Several examples are provided that illustrate different situations between proper  and compatible subspaces.
\end{abstract}

\bigskip



\section{Introduction}

Let $\e$ be a Banach space space which is continuously and densely included in a Hilbert space $\l$. 
A bounded operator  on $\e$ is a \textit{proper operator}
if it admits an adjoint with respect to the inner product of $\l$. This definition goes back
to  Gohberg and Zambicki\v{\i}  \cite{gz}, and it gives a simple condition under which they obtained several results on operators in spaces with two norms.  
In this context, we introduce the following class of subspaces: a  subspace $\s$ of $\e$ is called a \textit{proper subspace} if it is the range of a proper projection. If, in addition, the proper projection is self-adjoint with respect to the inner product of $\l$, then $\s$ is called a \textit{compatible subspace}. The aim of the present work is to study proper subspaces and their relation with compatible subspaces. 

The notion of compatible subspace has been studied in recent years. It  was in the paper \cite{cms01} by Corach, Maestripieri and Stojanoff, where the theory of compatibility was introduced and then studied systematically in the works \cite{cms02, cms05}. The usual setting to study problems concerning compatibility differs from our context. 
%It is a pair of Hilbert spaces where the inclusion between them might be not dense. 
One has a Hilbert space $(\h, \PI{ \,  }{   \, }_\h)$ and a positive semidefinite $A \in \b(\h)$, where $\b(\h)$ denote the algebra of bounded linear operators on $\h$. Then a bounded sesquilinear form can be defined by $\PI{f}{g}_A=\PI{Af}{g}_\h$, where $f,g \in \h$. If $\s$ is a closed subspace  of $\h$, the set of $A$-self-adjoint projections with range $\s$ is given by
$$ \mathcal{P}(A,\s)=\{ \, Q \in \b(\h) \, : \,Q^2=Q,\, AQ=Q^*A, \, R(Q)=\s \, \}.$$
The subspace $\s$ is compatible if $\mathcal{P}(A,\s)$ is not empty.  When $A$ is an injective operator, this is a special case of the setting described in the first paragraph, where $\e=\h$ and $\l$ the completion of $\h$ with respect to the norm induced  by  the inner product defined by $A$. In this case, if the set $\mathcal{P}(A,\s)$  is not empty, then it is a singleton. We remark that a definition of compatible subspace without assuming that $\e$ is a Hilbert space  was already considered in \cite[Remark 5.8]{cms05b}, but it was not studied in further works.

It is interesting to note that compatible subspaces can be found in the literature many years before.  At the time when this concept was not yet developed,  Sard  used an equivalent definition to give an operator theoretic approach to problems in approximation theory (see \cite{sard, cgm}). On the other hand, Hassi and Nordstr\"{o}m  \cite{hassinord} found conditions that guarantee the existence and uniqueness of self-adjoint projections with respect to an Hermitian form.  More recently, the notion of compatibility has been  related to different topics such as signal processing \cite{eldar, eldar wer}, frame theory \cite{acrs}, de Branges complementation theory \cite{ando90, debr, cms05b}, sampling theory \cite{ac, sz} and abstract splines \cite{cms02a, acg, boor, deutsch}.

Let us describe the contents of this paper. In Section \ref{notation} we establish notation and give the necessary background on proper operators. In Section \ref{section proper s} we prove elementary properties of proper subspaces. The set of all proper operators is an involutive Banach algebra, and thus, proper subspaces are ranges of projections in a Banach algebra.  We find equivalent conditions to describe proper subspaces in Theorem \ref{proper subsp}. One of these conditions asserts that a closed subspace $\s$ of $\e$ is a proper subspace if and only if there is another closed subspace $\t$ of $\e$ satisfying $$\s \, \dot{+} \, \t=(\s^{\perp} \cap \e) \, \dot{+}\,  (\t^{\perp} \cap \e)=\e,$$
where the orthogonal complement is considered with respect to $\l$. This kind of supplements  $\t$, which are also proper subspaces,  will be called \textit{proper companions} of $\s$. 

We address the question of when   a proper subspace is a compatible subspace in Section \ref{proper and comp compan}. Both notions coincide if the subspace has finite codimension, but they are different in general as we shall see in concrete  examples. In Theorem \ref{charac prop compatible s}  we obtain a criterion for a proper subspace to be compatible. Let $\s$ be a proper subspace and $\t$ a proper companion of $\s$, then the projection $P_{\s // \t}$ with range $\s$ and nullspace $\t$ is well-defined and continuous on $\e$. Our criterion basically asserts that $\s$ is compatible if and only if the operator
\[  C_{\s,\t}= P_{\s // \t} + P_{\s // \t} ^+ - I \]
has range equal to $\t \dot{+} (\t \cap \e)$. Here the symbol $+$ stands for the restriction to $\e$ of the adjoint in $\l$.  

We  prove in Theorem \ref{suff cond} different conditions equivalent to the compatibility of both a proper subspace and a fixed proper companion. Among other conditions, we find that a proper subspace $\s$ and a proper companion $\t$ are compatible subspaces exactly when the operator  $C_{\s, \t}$ is invertible on $\e$.  Next we examine when the compatibility of a proper companion $\t$ implies the compatibility of  other proper companion $\t_1$. As we shall show with examples in the next section, this property does not hold in general.  However, it holds in some special cases, for instance if the proper projections associated to a pair of companions are closed enough in a metric induced by the algebra of proper operators (Corollary \ref{two companions}). 
As a curious fact, we point out  that the existence of non compatible proper subspaces is closely related to spectral properties of symmetrizable operators (Corollary \ref{spec prop sym and prop}).

In Section \ref{section examples} we give several examples. In particular, if $\e$ is the space of trace class operators and 
$\l$ is the space of Hilbert-Schmidt operators, we provide  examples of non compatible proper subspaces  (Example \ref{prop no comp}). We also show  that the compatibility of a proper companion does not imply  compatibility of any other proper companion (Example \ref{two diff companions}). Finally, we exhibit examples of proper invertible operators.

%\footnote{With the exception of \cite[Remark 5.8]{cms05b} related to our setting}

\section{Preliminaries and notation}\label{notation}

Let $(\e, \,\| \, \,  \|_\e)$ be a Banach space contained in a Hilbert space $(\l, \, \| \, \,  \|_{\l})$. Denote by 
$\PI{\,  }{\,  }$ the inner product of $\l$. We assume that   
the inclusion $\e\hookrightarrow \l$ is continuous with  dense range. In order to 
simplify some computations, we further suppose that $\|f\|_{\l}\leq \|f \|_{\e}$ for all 
$f \in \e$.

\begin{rem}\label{def de J}
The Banach space $\e$ is continuously and densely contained in some Hilbert space $\l$ if and only if there exists a bounded conjugate-linear operator $J:\e \to \e^*$ such that $(Jf)(f)>0$ for all $f \in \e$, $f \neq 0$. If this condition is fulfilled, $\l$ is the Hilbert space obtained as the completion of $\e$ with respect to the norm $\|f\|_{\l}=((Jf)(f))^{1/2}$ and the inner product is given by the continuous extension of $\PI{f}{g}=(Jg)(f)$, where $f,g \in \e$.
\end{rem}

\subsection{Subspaces and projections.} Let $\b(\mathcal{\e})$ denote the algebra of bounded linear operators on  $\e$.  The range of an operator $T \in \b(\e)$ is denoted by $R(T)$, and its nullspace by $N(T)$. An operator $T \in \b(\e)$ is a projection if $T^2=T$. 
We denote by $\s \, \dot{+} \, \t$ the direct sum of two subspaces $\s$ and $\t$ of $\e$.  If these subspaces are closed and  $\s \, \dot{+} \, \t=\e$, the oblique projection $P_{\s // \t}$ onto $\s$ along $\t$ is the bounded projection with range $\s$ and nullspace $\t$.
Given a subset $\s$ of $\e$,  $\s^{\perp}$ is the usual orthogonal complement as a subspace of $\l$, that is
$$\s^{\perp}=\{ \, f \in \l \, : \,  \PI{f}{g}=0, \, \forall \, g \in \s  \,  \}.$$ 
It is easily seen that $\s^{\perp} \cap \e$ is a closed subspace of $\e$. Moreover, we have $\s^{\perp} \cap \e=J^{-1}(\s^{\circ})$, where $J$ is the map defined in Remark \ref{def de J} and $\s^{\circ}$ is the annihilator of $\s$.  

Throughout, the closure  $\overline{\s}$ of a closed subspace $\s$ of $\e$ is understood with respect to the topology of $\l$. The operator $P_{\overline{\s}}$ is the orthogonal projection onto $\overline{\s}$. It will be useful to state here the following result on projections.

\begin{teo}[Ando \cite{ando}]\label{formula proj}
Let $\s$ and $\t$ be two closed subspaces of a Hilbert space $\l$. If $\s \dot{+} \t=\l$, then the operator
$P_\s - P_\t$ is invertible and 
$$(P_\s - P_\t)^{-1}=P_{\s // \t} + P_{\s // \t} ^* -I  , \, \, \, \, \, \, \, \, \, \, \, P_{\s // \t}=P_\s (P_\s - P_\t)^{-1}. $$
\end{teo}
\noi We remark that the first formula above was first proved by Buckholtz \cite{buck}.

\subsection{Proper operators} 
 %The notion of proper operator was coined by Gohberg and Zambicki\v{\i} in \cite{gz}. 
In this subsection, we describe the basic properties of proper operators proved in \cite{gz}. 
 An operator $T \in \b(\e)$ is a \textit{proper operator}  if and only if for every $f \in \e$, there is a vector $g \in \e$ such that $\PI{Th}{f}=\PI{h}{g}$ for all $h \in \e$. This allows to define $T^+f=g$, and it can be shown that $T^+ \in \b(\e)$.

\begin{teo}[Gohberg-Zambicki\v{\i} \cite{gz}]\label{proper op}
Let $T$ be a proper operator. Then following statements hold:
\begin{enumerate}
	\item[i)] $T$ has a bounded extension $\bar{T}$  on $\l$. The usual operator norms of $\bar{T} \in \b(\l)$ and $T \in \b(\e)$ are related by  $$\| \bar{T}\|_{\b(\l)}\leq \min \{  \,  \| T^+T\|_{\b(\e)}, \, \|TT^+\|_{\b(\e)} \, \}.$$
	\item[ii)] If $\sigma_\e(T)$ and $\sigma_\l(\bar{T})$ denote the spectrum of $T$ on $\e$ and the spectrum of $\bar{T}$ on $\l$, respectively, then 
	$$\sigma_\l(\bar{T}) \subseteq \sigma_\e(T) \cup \overline{\sigma_\e(T^+)},$$ 
	where the last bar indicates complex conjugation.
  \item[iii)] If $T$ is a compact operator on $\e$, then $T$  is a compact operator on $\l$. Moreover, $\sigma_\l(\bar{T})=\sigma_\e(T)$ and the eigenspaces of $T$  in  $\e$ and $\l$ corresponding to the non zero eigenvalues  coincide.
\end{enumerate}
\end{teo}

When $T$ is a proper operator, it turns out that $T^+=\bar{T}^*|_E$, where the last adjoint is the  adjoint of $\bar{T}$ with respect to the inner product of $\l$. A \textit{symmetrizable operator} is a proper operator $T$ such that $T^+=T$. This class of operators   was studied independently by  Dieudonn\'e \cite{dieudonne}, Krein \cite{krein b} and Lax \cite{lax}.

We denote by $\p$ the set of all proper operators. It is not difficult to see that
$\mathfrak{P}$ is not closed in $\b(\e)$. However, $\p$ becomes an involutive unital Banach algebra   with  the norm defined by
\[  \|T\|_{\p} := \| T \|_{\b(\e)} + \|T^+\|_{\b(\e)} \,. \]

\begin{rem}\label{notions of spec} 
There are three different notions of groups of invertible operators: the group of invertible operators on $\e$, the group of invertible operators on $\l$ and the group of invertible operators of the algebra $\p$. These groups are denoted by $Gl(\e)$, $Gl(\l)$ and $\p^{\times}$, respectively.   If $T$ is a proper operator, we write $\sigma_\e(T)$ for the spectrum in the Banach space $\e$, $\sigma_\l(\bar{T})$ for the spectrum of the continuous extension $\bar{T}$ in the Hilbert space $\l$ and  $\sigma_{\p}(T)$ for spectrum  in the Banach algebra $\p$.  The relation between the first two notions of spectrum is stated in Theorem \ref{proper op}. Since $T$ is invertible in the algebra $\p$ if and only if $T$ and $T^+$ are invertible on $\e$, one can see that 
$$ \sigma_\p (T)=\sigma_\e(T) \cup \overline{\sigma_\e(T^+)}.  $$ 
 There are examples which show that the inclusion $\sigma_\e(T)\subseteq \sigma_\p(T)$ may be strict.  
\end{rem}

\section{Proper subspaces}\label{section proper s}

A closed subspace $\s \subseteq \e$ is called a $\textit{proper subspace}$ if there exists a proper projection $Q$ such that $R(Q)=\s$.
In this section we prove  basic facts on proper subspaces. We start with two  examples of proper subspaces.

\begin{ejem}\label{dim rgo finita}
Let $\s$ be a finite dimensional subspace of $\e$. We can construct a basis $\{ \, f_1 , \ldots, f_n \, \}$ of $\s$ satisfying $\PI{f_n}{f_m}=\delta_{nm}$. In fact, note that we only have to apply the Gram-Schmidt process to any basis of $\s$ to get a new basis with this property.
On the other hand, it is well known that as an operator on $\l$, 
any  projection $Q$ onto $\s$ can be written as 
\begin{equation}\label{proj finito}
  Q= \sum_{i=1}^n  \PI{\, \cdot \,}{h_n}f_n \, ,    
\end{equation}  
where $\{ \, h_1 \, , \, \ldots \, h_n \}$ is an orthonormal set satisfying $\PI{f_n}{h_n}=\delta_{nm}$. If we restrict this projection to $\e$, we find a characterization of an arbitrary proper projection $Q$ with finite dimensional range: $Q$ is a proper projection if and only if $h_1 \, , \,  \ldots \, , \, h_n \in \e$. Furthermore, by choosing $h_i=f_i$, $i=1, \, \ldots \, , n$, we have proved that any finite dimensional subspace  is a proper subspace.    
\end{ejem}

\begin{ejem}
Let $T$ be a proper operator. Suppose that $\la$ be an isolated point in $\sigma_\p(T)$. For the different notions of spectrum of a proper operator see Remark \ref{notions of spec}. Let $B_\epsilon(\la)$ be the open ball of radius $\epsilon$ centered at $\la$. Assume that $\epsilon$ satisfies $B_{2\epsilon}(\la) \cap \sigma_\p(T)=\{ \, \la \, \}$. In particular, this implies that $B_{2\epsilon}(\la) \cap \sigma_\e(T)=\{ \, \la \, \}$ and $B_{2\epsilon}(\bar{\la}) \cap \sigma_\e(T^+)=\{ \, \bar{\la} \, \}$.
We claim that 
\[   Q= \frac{1}{2 \pi i} \int_{\partial B_\epsilon(\la)} (z-T)^{-1} \, dz \]
is a proper projection, and thus, $R(Q)$ is a proper subspace. 

To prove our claim, let $\gamma:[-\pi,\pi]\to \partial B_\epsilon(\la)$ be a smooth curve with the positive orientation. Pick a partition $0=t_0 < t_1 < \ldots <t_n= \pi$ of the interval $[0,\pi]$, and then, consider the partition $t_{-k} =- t_k$ of $[-\pi,0]$. For $n$ large enough, the above defined integral can be approximated by the Riemann sum
\[  \frac{1}{2\pi i} \sum_{i=-n}^n (\gamma(t_i)-T)^{-1} \dot{\gamma}(t_i) \Delta t_i . \]
On the other hand, if $z \in \partial B_\epsilon(\la)$, then $z-T$ and  $\bar{z}-T^+$ are invertible in $\e$. We can define the following projection
\[   P= \frac{1}{2 \pi i} \int_{\partial B_\epsilon(\bar{\la})} (\bar{z}-T^+)^{-1} \, dz. \]
Then the curve $\beta(t)=\overline{\gamma(-t)}$ is positive oriented, and $\beta([-\pi,\pi])=\partial B_\epsilon(\bar{\la})$. The projection $P$ can be approximated by 
\[   \frac{1}{2\pi i} \sum_{i=-n}^n (\beta(t_i)-T^+)^{-1} \dot{\beta}(t_i) \Delta t_i . \]
Next note 
\begin{align*}
\PI{ \frac{1}{2\pi i} \sum_{i=-n}^n (\gamma(t_i)-T)^{-1} \dot{\gamma}(t_i) \Delta t_i \, h}{f} 
& = -\PI{h}{ \frac{1}{2\pi i} \sum_{i=-n}^n (\overline{\gamma(t_i)}-T^+)^{-1} \overline{\dot{\gamma}(t_i)} \Delta t_i \,f} \\
& = \PI{h}{\frac{1}{2\pi i} \sum_{i=-n}^n (\beta({-t_i})-T^+)^{-1} \dot{\beta}(-t_i)\Delta t_i \, f} \\
& = \PI{h}{\frac{1}{2\pi i} \sum_{i=-n}^n (\beta({t_i})-T^+)^{-1} \dot{\beta}(t_i)\Delta t_i \, f} ,
\end{align*}
where in the last step we have used that the partition is symmetric with respect to the origin. 
Letting $n \to \infty$, we get $\PI{Qh}{f}=\PI{h}{Pf}$. Thus, $Q$ is a proper projection and $Q^+=P$. 
\end{ejem}

Now we prove some elementary properties of proper operators.

\begin{lem}\label{relaciones usuales}
Let $T$ be a proper operator, then
\begin{itemize}
\item[i)] $N(T^+)=R(T)^\perp \cap \e$.
\item[ii)] $\ov{R(T^+)}\cap \e=N(T)^\perp \cap \e$.
\end{itemize}
\end{lem}
\begin{proof}
$\text{i)}$ Let $f \in N(T^+)$. Then  $\PI{f}{Tg}=\PI{T^+f}{g}=0$ for all $g \in \e$, which means that $f \in R(T)^{\perp}\cap \e$. Conversely, suppose that $f \in R(T)^{\perp}\cap \e$. This is equivalent to  $0=\PI{f}{Tg}=\PI{T^+f}{g}$ for all $g \in \e$, that is,  $f \in N(T^+)$.

\smallskip

\noi $\text{ii)}$ Since $(T^+)^+=T$, we know from the first item that $N(T)=R(T^+)^\perp \cap \e$. Then take the orthogonal complement in $\l$ and the intersection with $\e$. 
\end{proof}

\begin{rem}\label{relaciones usuales proj}
Let $Q$ be a proper projection. As a special case of the first item in the above lemma,  we get that 
$N(Q^+)=R(Q)^{\perp} \cap \e$, and since $R(Q^+)=N((Q-I)^+)$, we also have $R(Q^+)=N(Q)^\perp \cap \e$. 
\end{rem}

%\begin{rem}\label{relaciones usuales proj}
%As a special case of the first item in the above lemma, if $Q$ is a proper projection, we get that 
%$N(Q^+)=R(Q)^{\perp} \cap \e$ and $R(Q^+)=N(Q)^\perp \cap \e$. 
%\end{rem}

\begin{rem}\label{prop transp}
It will be useful to rephrase the definition of proper operator in terms of range inclusions. Indeed, an operator $T \in \b(\e)$ is proper if and only if $R(T'J)\subseteq R(J)$, where $J$ is the operator defined in Remark \ref{def de J} and $T'$ is the transpose of $T$.  
\end{rem}

In the case where $\e=\h$ is a Hilbert space, there is an injective positive operator $A \in \b(\h)$ such that $(Jg)(f)=\PI{Af}{g}_\h$ for all $f,g \in \h$. Then an operator $T \in \b(\h)$ is proper
if and only if there exists an operator $W \in \b(\h)$ such that  $AT=W^*A$, where the last adjoint is the adjoint with respect to $\h$. 
In the case of projections,  there is another useful characterization in \cite[Lemma 2.1]{cms05}: 
a projection $Q$ is proper if and only  
$$  R(A)=(R(A)\cap R(Q^*)) \, \dot{+} \, (R(A) \cap N(Q^*)).$$
 This result can also be proved in our setting with the obvious modifications.  

\begin{lem}\label{proj range split}
Let $Q \in \b(\e)$ be a projection. Then $Q$ is proper if and only if
$$  R(J)=(R(J)\cap R(Q')) \, \dot{+} \, (R(J) \cap N(Q')).$$
\end{lem}
\begin{proof}
We state in Remark \ref{prop transp} that  $Q$ is proper if and only if $R(Q'J)\subseteq R(J)$. 
Clearly, this is equivalent to the condition $R(J)=R(Q'J) \, + \, R((I-Q')J)$. 

On the other hand, we claim that  $Q$
is proper if and only if $R(Q'J)=R(J)\cap R(Q')$. In fact, note that 
if $R(Q'J)=R(J)\cap R(Q')$, then $R(Q'J)\subseteq R(J)$, which implies that $Q$ is proper. Now  if we suppose that $Q$ is proper, then we have $R(Q'J)\subseteq R(J) \cap R(Q')$. But since $Q'$ is a projection, any functional  $\phi=Jf=Q'\phi$, where $f \in \e$ and $\phi \in R(Q')$, can be written as $\phi=Q'\phi=Q'Jf$.  Therefore we get $R(Q'J)=R(J)\cap R(Q')$. 

Applying the same argument to $I-Q$, we also find that $Q$ is a proper projection if and only if $R((I-Q')J)=R(J)\cap N(Q')$.
It follows  that a projection $Q$ is proper if and only 
if $R(J)=R(Q'J) \, + \, R((I-Q')J)=(R(J)\cap R(Q')) \, \dot{+} \, (R(J) \cap N(Q'))$. 
\end{proof}

We give a characterization of proper subspaces in the following result.

\begin{teo}\label{proper subsp}
Let $\s$ be a closed subspace of $\e$. 
The following conditions are equivalent:
\begin{itemize}
\item[i)] $\s$ is a proper subspace. 
\item[ii)] There exists a projection $Q \in \b(\e)$ such that $R(Q)=\s$ and
  $$ R(J)=(R(J)\cap R(Q')) \, \dot{+} \, (R(J) \cap N(Q')).  $$
\item[iii)] There exists a closed subspace $\mathcal{T}$ of $\e$
such that 
$$ \s \, \dot{+} \, \mathcal{T}= (\s^{\perp} \cap \e ) \, \dot{+} \, (\mathcal{T}^{\perp}\cap \e)=\e.$$
\end{itemize}
\end{teo}
\begin{proof}
\noi $\text{i)} \Leftrightarrow \text{ii)}$  This follows immediately from Lemma \ref{proj range split}. 

\smallskip

\noi $\text{i)} \Leftrightarrow \text{iii)}$ We suppose that $\s$ is a proper subspace. Then there is a proper projection $Q$ such that $R(Q)=\s$.  
According to Remark \ref{relaciones usuales proj}, we can take $\mathcal{T}=N(Q)$. In fact, we have  $N(Q^+)=\s^{\perp} \cap \e$ and $R(Q^+)=\mathcal{T}^{\perp} \cap \e$. Since $Q^+$ is a projection in $\b(\e)$, we get $(\s^{\perp} \cap \e )\dot{+}(\mathcal{T}^{\perp}\cap \e)=\e$.

Now assume that there is a closed subspace $\mathcal{T}$ satisfying $\s \dot{+} \mathcal{T}= (\s^{\perp} \cap \e )\dot{+} (\mathcal{T}^{\perp}\cap \e)=\e$. Then we can define the continuous projections  $Q=P_{\s // \t}$ and $P=P_{\t^{\perp}\cap \e // \s^{\perp}\cap \e}\,$.  Note that for any $h, f \in \e$, we have 
$$
\PI{Qh}{f} = \PI{Qh}{Pf + (I-P)f}=\PI{Qh}{Pf}=\PI{(I-Q)h + Qh}{Pf}=\PI{h}{Pf}.
$$
This shows that $Q$ and $P$ are proper operators and $Q^+=P$. Hence $\s$ is a proper subspace. 
\end{proof}

If $\s$ is a proper subspace, we have seen that there exists a closed subspace $\t$ of $\e$ 
such that $\s \dot{+} \mathcal{T}= (\s^{\perp} \cap \e )\dot{+} (\mathcal{T}^{\perp}\cap \e)=\e$. We refer to any such subspace $\t$ as a \textit{proper companion of  $\s$}.

\begin{coro}\label{ort proper}
If $\s$ is a proper subspace of $\e$ with a proper companion $\t$. Then $\s^{\perp}\cap \e$, $\t$ and $\t^{\perp}\cap \e$ are  proper subspaces.
\end{coro}
\begin{proof}
Let $Q$ and $P$ be the proper projections defined in the proof of Proposition \ref{proper subsp}. The ranges of 
$I-Q$, $P$ and $I-P$ are the subspaces $\t$, $\t^{\perp}\cap \e$ and $\s^{\perp} \cap \e$, respectively. Hence 
these three subspaces are proper.  
\end{proof}

\begin{coro}\label{proper caract}
Let $\s$ be a proper subspace of $\e$. Then the following assertions hold:
\begin{itemize}
\item[i)] If $\t$ is a proper companion  of $\s$, then $\bar{P}_{\s // \t}=P_{\ov{\s}//\ov{\t}}$. 
%\item[ii)] If $\t$ is a proper companion  of $\s$, then $\ov{\s} \, \dot{+} \, \ov{\t}=\l.$
\item[ii)] $\overline{\s}\cap \e =\s$.
\end{itemize}
\end{coro}
\begin{proof}
$\text{i)}$ First note that the bounded projection $Q:=P_{\s // \t}$ is well defined because $\s \, \dot{+} \, \t = \e$. In the proof of Theorem \ref{proper subsp} we have seen that $Q$ is a proper operator. Then, $Q$ has a bounded extension $\bar{Q}$ to the Hilbert space $\l$. 
Note that $\s \subseteq R(\bar{Q})\subseteq \ov{\s}$, and $\bar{Q}$ has closed range, which implies $R(\bar{Q})=\ov{\s}$. Similarly, one can check that $N(\bar{Q})=\ov{\t}$. 
 
\smallskip

\noi $\text{ii)}$ The nontrivial inclusion is $\overline{\s}\cap \e \subseteq \s$. Pick $f \in \overline{\s}\cap \e$. Since $\s$ is a proper subspace,
there is a proper projection $Q$ with range $\s$ and nullspace $\t$. By the first item, we know that $\bar{Q}=P_{\ov{\s}// \ov{\t}}$, so we 
get $f=\bar{Q}f=Qf \in \s$. 
\end{proof}

\begin{coro}\label{relacion supl propers}
Let $\t$ be a proper companion of a proper subspace $\s$ and $G \in \p^{\times}$, then 
$G(\s)$ and $G(\t)$ are proper companions. Moreover, if $\t_1$ is another proper companion of $\s$, then
there exists an operator $G \in \p^{\times}$ such that $G(\t)=\t_1$ and $G(\s)=\s$. 
\end{coro}
\begin{proof}
 For the first assertion, we only have to  note that the projection $P=GP_{\s // \t}G^{-1}$ is a proper operator with range 
$G(\s)$ and nullspace $G(\t)$. In order to show the second assertion, consider the bounded operator
\[ G_0= (P_{\t_1 // \s})|_{\t} : \t \to \t_1 . \]
It is easy to check that $G_0$ is an isomorphism. Then the operator defined by 
$$G(f_1 + f_2)=f_1 + G_0f_2,  \, \, \, \, f_1 \in \s,\, f_2 \in \t,$$   
is invertible on $\e$, $G(\t)=\t_1$ and $G(\s)=\s$. To show that $G$ is a proper operator,  we note that it can be rewritten as
$$  G=P_{\s // \t} + P_{\t_1 // \s} \, P_{\t// \s}. $$
Since each projection in this expression is a proper operator, we get that $G$ is a proper operator  and 
$$ G^+= P_{ \t^{\perp}\cap \e // \s^{\perp}\cap \e}+ P_{\s^{\perp}\cap \e// \t ^{\perp} \cap \e} \, P_{\s^{\perp}\cap \e// \t_1 ^{\perp} \cap \e}.$$
 What remains is to prove that $G$ is invertible in the Banach algebra $\p$. Since $G$ is invertible on $\e$, we have to show that $G^+$ is invertible  on $\e$. Clearly, $G^+$ is injective. To see that $G^+$ is surjective, given $g \in \e$, write $g=g_1 + g_2$, where $g_1 \in \s^{\perp}\cap \e$ and $g_2 \in \t^{\perp} \cap \e$. Then note that one can  also  write $g_2=g_{2,1} + g_{2,2}$, where 
$g_{2,1} \in \s^{\perp}\cap \e$    and $g_{2,2} \in \t_1^{\perp} \cap \e$. Therefore, the vector $f=g - g_{2,1}$ satisfies $G^+f=g$. 
\end{proof}

\section{Proper and compatible subspaces}\label{proper and comp compan}

A closed subspace $\s \subseteq \e$ is called a $\textit{compatible subspace}$ if there exists a proper projection  $Q$ such that $Q=Q^+$ and $R(Q)=\s$.   The following elementary characterizations of compatible subspaces will be useful later.

%\begin{lem}
%Let $Q$ be a proper projection, then $Q=Q^+$ if and only if $N(Q)\subseteq R(Q)^{\perp} \cap \e$. 
%\end{lem}

\begin{lem}\label{comp caract}
Let $\s$ be a closed subspace of $\e$. The following conditions are equivalent:
\begin{itemize}
\item[i)] $\s$ is compatible.
\item[ii)] $\s \, \dot{+} \, (\s^{\perp} \cap \e)= \e$.
\item[iii)]  There exists a proper projection $Q$ such that $R(Q)=\s$ and $N(Q)\subseteq \s^{\perp} \cap \e$. 
\item[iv)] $R(J)=J(\s) \, \dot{+} \,  (R(J) \cap \s^{\circ}) $.
\item[v)] $\ov{\s}\cap \e=\s$ and $P_{\ov{\s}}(\e)\subseteq \e$.  
\end{itemize}
\end{lem}
\begin{proof}
$\text{i)} \Leftrightarrow \text{ii)}$ Suppose that $\s$ is compatible subspace. Then there is a proper projection $Q$ such that $Q=Q^+$ and $R(Q)=\s$. Using Remark \ref{relaciones usuales proj}, we get $N(Q)=N(Q^+)=\s^\perp \cap \e$, which yields $\e=R(Q) \dot{+} N(Q) = \s \dot{+} (\s^\perp \cap \e)$. To prove the converse, now assume that $\s \, \dot{+} \, (\s^{\perp} \cap \e)= \e$. Then the projection $Q=P_{\s// \s^\perp \cap \e}$ is continuous on $\e$, and note that $\PI{Qh}{f}=\PI{Qh}{Qf}=\PI{h}{Qf}$ for all $f,h \in \e$. Thus, $Q$ is a proper projection, $R(Q)=\s$ and $Q^+=Q$.

\smallskip

\noi $\text{i)} \Leftrightarrow \text{iii)}$  This is a direct consequence of a result of Krein \cite{krein a}. We refer to \cite[Lemma 2.5]{cms05} for a proof when $\e$ is a Hilbert space. It is not difficult to see that this proof  can also be  carried out in the Banach setting.

\smallskip

\noi $\text{i)} \Leftrightarrow \text{iv)}$ We only need to follow the proof in \cite[Prop. 2.14 2.]{cms05}, taking into account the map $J$ that shows up in the Banach setting.

\smallskip

\noi $\text{i)} \Leftrightarrow \text{v)}$ This was proved in  \cite[Prop. 3.5]{anchdi} in the setting of Hilbert spaces. The proof in our setting goes exactly the same line.
\end{proof}

\begin{rem}
If $\s$ is a compatible subspace, there exists a unique proper projection $Q$ such that $Q^+=Q$ and $R(Q)=\s$. This follows immediately 
from the fact that $Q$ is uniquely determined as $(P_{\ov{\s}})|_{\e}$. We denote this projection by $Q_\s$. 
\end{rem}

Note that in Example \ref{dim rgo finita} we actually show that every finite dimensional subspace is compatible. Subspaces of finite codimension in $\e$ may not be compatible or proper, but both notions coincide for this type of subspaces. 

\begin{prop}\label{equi cod fin}
Let $\s$ be a closed subspace of $\e$ with finite codimension. Then $\s$ is a proper subspace if and only if $\s$ is a compatible subspace. 
\end{prop}
\begin{proof}
The ``if" part is trivial. To prove the ``only if" part, suppose  that $\s$ is a proper subspace. Any supplement of $\s$ in $\e$ has to be finite dimensional. In particular, a proper companion $\t$ is finite dimensional, and then $\ov{\t}=\t$. Let $Q$ be a proper projection such that $R(Q)=\s$. Then $P=I-Q$ is a proper projection with range $\t$. In Example \ref{dim rgo finita} we see that $P$ can be described by formula (\ref{proj finito}). Note that  $R(P)=\t=\text{span} \{ \,   f_1 \, , \, \ldots \, , \, f_n      \,  \}$ and $R(P^+)=N(P)^{\perp}\cap \e=\s^{\perp} \cap \e=\text{span} \{ \, h_1 \, \ldots  \, , \, h_n \, \}$. From these facts, we get $\dim \t=\dim \s^{\perp} \cap \e$.

On the other hand, recall that  $\ov{\s} \, \dot{+} \, \t=\l$ by Corollary \ref{proper caract}. Since $\s^{\perp}$ is  a supplement for $\ov{\s}$ in $\l$, it follows that $\dim \t= \dim \s^{\perp}$. Therefore, we obtain that $\dim \s^{\perp}=\dim \s^{\perp} \cap \e$. Hence $\s^{\perp}=\s^{\perp} \cap \e$.

To prove that $\s$ is compatible, it suffices to show that  $\overline{\s}\cap \e =\s$ and $P_{\ov{\s}}(\e)\subseteq \e$  by Lemma \ref{comp caract}. According to  Corollary \ref{proper caract}, we have $\overline{\s}\cap \e =\s$. 
To prove the second condition,  we use that $\s^{\perp}=\s^{\perp} \cap \e \subseteq \e$. We thus get $P_{\ov{\s}}(\e)=(I-P_{\s^{\perp}})(\e)=(I-P_{\s^{\perp} \cap \e})(\e)\subseteq \e$.  This completes the proof.
\end{proof}

\begin{ejem}\label{no prop no comp}
Given a vector $g \in \l \backslash \e$, $\|g\|_\l=1$, the subspace 
$$\s=\{ \, f \in \e \, : \, \PI{f}{g}=0 \,  \}=\{ \, g \, \}^{\perp} \cap \e$$
is neither a compatible  nor a proper subspace. Note that $\s$ has finite codimension, so it is enough to prove that 
$\s$ is not compatible by Proposition \ref{equi cod fin}. The first condition in Lemma \ref{comp caract} \text{v)}, that is, $\ov{\s}\cap \e =\s$, clearly holds for this subspace.  On the other hand, the orthogonal projection onto $\ov{\s}$ is given by 
$$  P_{\ov{\s}}(f)=f - \PI{f}{g}g.$$
Apparently,   it follows that $P_{\ov{\s}}(\e) \not\subseteq \e$ by our choice of the function $g$. Hence $\s$ is not compatible. 
\end{ejem}

\begin{ejem}
In \cite[Example 4.3]{cms02} the authors gave the following  example of a non compatible subspace. 
Let $A$ be a positive injective non invertible operator acting on $\e=\h$. As usual,  $\l$ is the Hilbert space obtained by completing $\h$ with respect to the inner product $\PI{f}{h}_A=\PI{Af}{h}_\h$.  Pick any vector $g \in \h \setminus R(A)$. Then 
they proved that the subspace $\s=\{ \, g \, \}^{\perp_{\h}}$ is not compatible. 
Furthermore, the subspace $\s$ has finite codimension in $\h$. Thus, by Proposition \ref{equi cod fin}, we also know 
that $\s$ is not a proper subspace.
\end{ejem}

We shall need the following algebraic result by Maestripieri. Its proof can be found in \cite[Prop. 2.8]{acg}. 

\begin{lem}\label{algebraic lem}
Let $T_1 , T_2 \in \b(\e)$  such that $R(T_1)\cap R(T_2)=\{0\}$. 
Then
$
\e=N(T_1)+ N(T_2)
$
if and only if
$
R(T_1)+ R(T_2) =R(T_1+T_2).
$
\end{lem}

\begin{rem}\label{alg lem rem}
It will be useful to state the last condition of the above lemma in a slightly different way. We notice  that $R(T_1)+ R(T_2) =R(T_1+T_2)$ if and only if $R(T_1) \subseteq R(T_1+T_2)$ (see \cite[Prop. 2.4]{acm15}).  
%In fact, we observe that $R(T_1)+ R(T_2) =R(T_1+T_2)$ obviously implies the inclusion $R(T_1) \subseteq R(T_1+T_2)$. To show the converse,  if $f_1,f_2 \in \e$, then there is some $g \in \e$ such that  $T_1f_1 + T_2 f_2=(T_1 + T_2)g + T_2f_2=(T_1 + T_2)(g+f_2) - T_1f_2$. Using again that $R(T_1) \subseteq R(T_1+T_2)$, we find another vector $h \in \e$ such that $T_1f_2 =(T_1 + T_2)h$. We thus get $T_1f_1 + T_2 f_2=(T_1 +T_2)(g + f_2 - h)$, which proves our claim. 
\end{rem}

Let $\s$ be a proper subspace of $\e$ and $\t$ a proper companion. As we shall see, the compatibility
of these subspaces is related to properties of the following symmetrizable operator 
\[  \c=P_{\s // \t} + P_{\s // \t}^+ - I.  \]
We observe that its extension $\cl$ is invertible on $\l$ if and only if $\bar{V} + \bar{V}^*$ is invertible on $\l$,
where $V=2P_{\s // \t}-I$. Since $V$ is a symmetry, this is equivalent to  $-1 \notin \sigma_\l(\bar{V}^*\bar{V})$, 
which clearly holds since $\bar{V}^*\bar{V}$ is positive on $\l$. Thus, $\cl$ is invertible on $\l$. In particular, $\c$ is 
injective as an operator on $\e$.

Our main result to decide when a proper subspace is compatible now follows. 
Its proof is based on Lemma \ref{algebraic lem}. This idea has been used in \cite[Prop. 2.9]{acg} to relate compatible subspaces in Hilbert spaces and Bott-Duffin inverses.  

\begin{teo}\label{charac prop compatible s}
Let $\s$ be a proper subspace of $\e$ and $\t$ a proper companion of $\s$. The following 
assertions are equivalent:
\begin{itemize}
\item[i)] $\s$ is a compatible subspace. 
\item[ii)] $  \t \, \dot{+} \,(\t^{\perp} \cap \e) = R(\c)$.
\item[iii)] $\t^{\perp} \cap \e \subseteq R(\c)$.
\item[iv)]  $\t \subseteq R(\c)$.
\end{itemize}
If any of these statements is satisfied, the unique proper projection $Q_\s$ such that $Q_\s=Q_\s^+$ and $R(Q_\s)=\s$ is given by
$$
Q_\s=\c^{-1} \, P_{\s // \t}^+.
$$
\end{teo}
\begin{proof}
 $\text{i)} \Leftrightarrow \text{ii)}$ Set $Q= P_{\s // \t}$. We shall use  Lemma \ref{algebraic lem} with $T_1=Q^+$ and $T_2=Q-I$. Note that
$R(Q^+)=\t^\perp\cap \e$ and $R(Q-I)=\t$ have trivial intersection, and thus the lemma applies.  Then, as it is shown before
$$
N(Q^+)=R(Q)^\perp\cap \e=\s^\perp\cap \e  ; \, \, \, \,  N(Q-I)=R(Q)=\s.
$$
According to Lemma \ref{comp caract}, the fact that $\s$ is compatible is equivalent to $\e=\s + (\s^\perp \cap \e)$.

Clearly, the equivalence between $\text{ii)}$, $\text{iii)}$ and $\text{iv)}$ follows from Remark \ref{alg lem rem}.

Now we assume that $\s$ is compatible. Before the statement of this theorem, we have shown that the operator $\c=Q+Q^+-I$ is injective. Since we know that $R(Q^+)=\t^\perp \cap \e$, by  condition $\text{iii)}$ the operator   $(Q+Q^+-I)^{-1}Q^+$ is everywhere defined in $\e$. Apparently, it has closed graph: let $f_n\to f$ in $\e$ with $(Q+Q^+-I)^{-1}Q^+f_n\to g$. Then $Q^+f_n\to (Q+Q^+-I)g$. Also $Q^+f_n\to Q^+f$. It follows that 
$$
(Q+Q^+-I)g=Q^+f, \ \ \hbox{ i.e. } g=(Q+Q^+-I)^{-1}Q^+f.
$$
Thus, the operator $(Q+Q^+-I)^{-1}Q^+$ is bounded. We claim that $(Q+Q^+-I)^{-1}Q^+=Q_\s$. This is equivalent to proving that
$Q^+=(Q+Q^+-I)Q_\s$. Since $R(Q)=R(Q_\s)=\s$, one has $QQ_\s=Q_\s$, and thus
$$
(Q+Q^+-I)Q_\s=Q^+Q_\s.
$$
Note also that $Q^+(I-Q_\s)=0$: $R(I-Q_\s)=N(Q_\s)=\s^\perp\cap \e=N(Q^+)$. Then 
\[ Q^+Q_\s=Q^+(Q_\s +(I-Q_\s))=Q^+.\qedhere \]
\end{proof}

Of course,  the operator $\c$ may be not invertible on $\e$, and $\s$ can be a compatible subspace (see Example \ref{prop no comp}). In fact, $\c$ is invertible on $\e$ exactly when  $\s$ and $\t$ are both compatible subspaces.

\begin{teo}\label{suff cond}
Let $\s$ be a proper subspace of $\e$ and $\t$ a proper companion of $\s$. The following conditions 
are equivalent:
\begin{itemize}
\item[i)] $\s$ and $\t$ are  compatible subspaces. 
\item[ii)] $(P_{\ov{\s}} +  P_{\ov{\t}})(\e)\subseteq \e$.
\item[iii)] $(P_{\ov{\s}} -  P_{\ov{\t}})(\e)\subseteq \e$.
\item[iv)] $\c$ is invertible on $\e$.
\end{itemize}
\end{teo}
\begin{proof}
$\text{i)} \Rightarrow \text{ii)}$ This implication follows from Lemma \ref{comp caract} $\text{v)}$.

\smallskip

\noi $\text{ii)} \Leftrightarrow \text{iii)}$ Since $\s$ is a proper subspace,  we know that $\ov{\s} \, \dot{+} \, \ov{\t}=\l$ by Corollary \ref{proper caract} $i)$. Then, the following formula (see Theorem \ref{formula proj}) 
for the projection on a Hilbert space with range $\ov{\s}$ and nullspace $\ov{\t}$ can be used:
\[   P_{\ov{\s} // \ov{\t}}=P_{\ov{\s}} \, (P_{\ov{\s}} -  P_{\ov{\t}})^{-1}. \]
Equivalently, $P_{\ov{\s} // \ov{\t}} \,(P_{\ov{\s}} -  P_{\ov{\t}}) = P_{\ov{\s}}$. Interchanging the roles of the subspaces, we can also get
that $P_{\ov{\t} // \ov{\s}}\, (P_{\ov{\t}} -  P_{\ov{\s}}) = P_{\ov{\t}}$. Then, we obtain
\begin{equation}\label{symm}
 P_{\ov{\s}} - P_{\ov{\t}} =      (2P_{\ov{\s} // \ov{\t}}- I)  \, (P_{\ov{\s}} + P_{\ov{\t}}).  
\end{equation}
Since $\s$ is a proper subspace, we have $\bar{P}_{\s // \t}=P_{\ov{\s} // \ov{\t}}$ by Corollary \ref{proper caract} $i)$. Then, the symmetry $2P_{\ov{\s} // \ov{\t}}- I$ acting on $\l$  is an extension of the symmetry $2P_{\s // \t}-I$, which is an invertible operator on $\e$. From Eq. (\ref{symm}), it is now clear that  $(P_{\ov{\s}} - P_{\ov{\t}})(\e)\subseteq \e$ if and only if $(P_{\ov{\s}} + P_{\ov{\t}})(\e)\subseteq \e$.

\smallskip

\noi $\text{iii)} \Rightarrow \text{i)}$: We have shown that  $(P_{\ov{\s}} - P_{\ov{\t}})(\e)\subseteq \e$ is equivalent to $(P_{\ov{\s}} + P_{\ov{\t}})(\e)\subseteq \e$. If we add or subtract $P_{\ov{\s}} - P_{\ov{\t}}$ and $P_{\ov{\s}} + P_{\ov{\t}}$, we prove this implication. 

\smallskip

\noi $\text{iii)} \Leftrightarrow \text{iv)}$ Set $Q:=P_{\s//\t}$. By Corollary \ref{proper caract} $i)$, we have $\ov{\s} \, \dot{+} \, \ov{\t}=\l$. Therefore $P_{\ov{\s}} - P_{\ov{\t}}$ is invertible on $\l$, and its inverse is given by 
$$
(P_{\ov{\s}} - P_{\ov{\t}})^{-1}=\bar{Q} + \bar{Q}^* - I.
$$
These facts can be found again in Theorem \ref{formula proj}. If the operator $Q + Q^+ -I$ is invertible on $\e$, then its extension $\bar{Q} + \bar{Q}^* - I$ to 
$\l$ maps $\e$ onto $\e$. Therefore we get $(P_{\ov{\s}}-P_{\ov{\t}})(\e)=(P_{\ov{\s}}-P_{\ov{\t}})(Q + Q^+ -I)(\e)=\e$. 
To prove the converse, we assume that $(P_{\ov{\s}}-P_{\ov{\t}})(\e)\subseteq \e$. Note  that 
$(P_{\ov{\s}}-P_{\ov{\t}})^{-1}(\e)=(\bar{Q} + \bar{Q}^* - I)(\e)=(Q + Q^+ -I)(\e)\subseteq \e$, which implies that  $\e \subseteq (P_{\ov{\s}}-P_{\ov{\t}})(\e)$. Therefore $(P_{\ov{\s}}-P_{\ov{\t}})(\e)=\e$. Now we have
$(Q+Q^+-I)(\e)=(\bar{Q} + \bar{Q}^* - I)(\e)=(\bar{Q} + \bar{Q}^* - I)(P_{\ov{\s}}-P_{\ov{\t}})(\e)=\e$.
\end{proof}

Let $\s$ be a compatible subspace of $\e$. By Corollary \ref{relacion supl propers} any proper companion of $\s$ arises as the image of other proper companion by an invertible operator in the Banach algebra $\p$ given by the proper operators. In general, this invertible operator does not extend to a unitary operator on $\l$. Thus, if $\t$ and
$\t_1$ are two proper companions of $\s$, and $\t$ is a compatible subspace, the subspace $\t_1$ may be not compatible. For a concrete example of this situation see Example \ref{two diff companions}. However, we shall give below two sufficient conditions to ensure the compatibility of $\t_1$.  We first have to introduce the following metric in the set of all proper companions of $\s$:
$$  d(\t_1,\t_2)=\| P_{\t_1//\s} - P_{\t_2//\s} \|_\p \, , $$
where $\t_i$, $i=1,2$, are proper companions of $\s$ and $\| \, \|_\p$ is the norm of the algebra $\p$. 

%The essential idea is to  use that $C_{\s,\t_1}$ is invertible
%Theorem \ref{suff cond} $\text{iv)}$
%together with the fact that the set of invertible operators is open in a Banach algebra to obtain that 
%For the second condition, we need some more information on the spectrum of the operator $\c$. It turns out that this operator is not invertible
%on $\e$ exactly when zero  belongs to its essential spectrum on $\e$.   

\begin{coro}\label{two companions}
Let $\s$ be a proper subspace of $\e$ and $\t$ a proper companion of $\s$. Suppose that 
$\s$ and $\t$ are compatible subspaces. The following assertions hold:
\begin{itemize}
\item[i)] There exists a constant $r>0$, depending only on $\s$ and $\t$, such that $\t_1$ is a compatible subspace whenever $d(\t,\t_1)<r$.
\item[ii)] Let $G \in \p^{\times}$ such that $G-I$ and $G^+-I$ are compact operators on $\e$, then $G(\s)$ and $G(\t)$ are compatible subspaces.  
\end{itemize}
\end{coro}
\begin{proof}
\noi $\text{i)}$  In order to prove our assertion, it is enough to show that the map
\begin{equation}\label{cont subs}
  \{ \, \t_1 \subseteq \e \, : \, \t_1 \text{ is a proper companion of } \s \, \} \to \b(\e), \, \, \, \, \t_1 \mapsto C_{\s,\t_1} 
\end{equation}	
is continuous at $\t$, when the first space is endowed with the metric $d$ defined above  and $\b(\e)$ is considered with its usual operator norm $\| \, \|$. In fact, note that if this map is continuous, then there is constant $r>0$ depending on $P_{\t // \s}$, such that
$$ \| C_{\s, \t_1} - \c \| \leq 1/ \| \c^{-1} \|,  $$
whenever $d(\t, \t_1)<r$. This latter inequality above implies that $C_{\s, \t_1}$ is invertible on $\e$, and by Theorem \ref{suff cond}, this is equivalent to the compatibility of $\s$ and $\t_1$.

Let  $(\t_n)$ be  proper companions of the compatible subspace $\s$ such that $d(\t_n,\t)\to 0$.
This means that 
\begin{equation}\label{converg} 
\| P_{\t_n//\s} - P_{\t // \s} \| \to 0  \, \, \text{ and } \, \, \| P_{\s^\perp \cap \e//\t_n ^\perp \cap \e} - P_{\s^\perp \cap \e//\t^\perp \cap \e} \| \to 0.   
\end{equation}
On the other hand, we recall from Corollary \ref{relacion supl propers} that there exist  operators $G_n \in \p^{\times}$ such that $G_n(\t)=\t_n$ and $G(\s)=\s$. In particular, it follows that $G_nP_{\s//\t}G_n^{-1}=P_{\s//\t_n}$. As we have shown in the proof of this corollary, $G_n$ and $G_n^+$ are given by 
$$ G_n=P_{\s // \t} + P_{\t_n // \s} \, P_{\t// \s} $$
and
$$ G^+_n= P_{ \t^{\perp}\cap \e // \s^{\perp}\cap \e}+ P_{\s^{\perp}\cap \e// \t ^{\perp} \cap \e} \, P_{\s^{\perp}\cap \e// \t_n ^{\perp} \cap \e}.$$  
Using Eq. (\ref{converg}), we see that $\|G_n-I\| \to 0$ and $\|G_n^+ -I \| \to 0$. Therefore, we obtain that 
$$ \| C_{\s, \t_n} - \c \|= \| G_n P_{\s // \t}G_n^{-1} + (G_n^+)^{-1} P_{\s // \t}^+G_n^+ - P_{\s // \t} - P_{\s // \t}^+ \| \to 0. $$
This completes the proof of the continuity of the map defined in (\ref{cont subs}). 

\smallskip

\noi \text{ii)}  We set $\t_1=G(\t)$, $\s_1=G(\s)$ and $Q=P_{\s//\t}$. Then  we note that 
\begin{align*}
C_{\s_1 , \t_1} & = GQG^{-1} + (G^+)^{-1}Q^+G^+ -I \\ 
& = (G-I)QG^{-1} + Q(G^{-1}-I) + ((G^+)^{-1}-I)Q^+G^+ + (G^+)^{-1}Q^+(G^+-I)+\c \\
& = K + \c \,,
\end{align*}
for some compact operator $K$ on $\e$. This implies that the essential spectrum $\sigma_{ess,\e}(C_{\s_1 , \t_1})$ of $C_{\s_1 , \t_1}$ coincides with that of $\c$. Thus, we get that $0 \notin \sigma_{ess,\e}(C_{\s_1 , \t_1})$.

Now we recall that the essential spectrum of a symmetrizable operator consists of those numbers in the spectrum over $\e$ which are not isolated eigenvalues of finite multiplicity (see \cite[Thm. 1]{nieto}). Applying this result to the operator $C_{\s_1 , \t_1}$, we have that its spectrum can be written as following  disjoint union:
$$  \sigma_\e(C_{\s_1 , \t_1})=\sigma_{ess,\e}(C_{\s_1 , \t_1}) \cup \sigma_{p,\e}(C_{\s_1 , \t_1}), $$
where $\sigma_{p,\e}(C_{\s_1 , \t_1})$ is the point spectrum of $C_{\s_1,\t_1}$ on $\e$ consisting on isolated eigenvalues of finite multiplicity. Since $\bar{C}_{\s_1,\t_1}$ is always invertible on $\l$, we have that zero does not belong to its point spectrum on $\l$. 
But  the point spectrum  of a symmetrizable operator coincides with the point spectrum of its extension (see \cite{krein b, lax}). Therefore, $0 \notin \sigma_{p,\e}(C_{\s_1 , \t_1})$. Hence $C_{\s_1 , \t_1}$ is invertible on $\e$, and thus, $\s_1$ and $\t_1$ are compatible subspaces.   
\end{proof}

As it is shown in Example \ref{prop no comp}, there are proper subspaces which are not compatible subspaces. 
%According to  Example \ref{dim rgo finita} and  Proposition \ref{equi cod fin},
 These kind of subspaces must have 
infinite dimension and infinite codimension. 
Now we shall see that each proper projection onto a proper subspace which is not compatible gives rise a symmetrizable operator with non real  spectrum as an operator in $\e$. Up to best of our knowledge, the first example of such kind of symmetrizable operator was constructed in \cite{dieudonne}. Other examples were given in \cite{gz, barnes00, achl}. All of them rely on a fundamental result by Krein on the spectrum of Toeplitz matrices (see \cite[Thm. 13.2]{krein}).

\begin{coro}\label{spec prop sym and prop}
Let $\s$ be a proper subspace of $\e$ which is not a compatible subspace. Let $Q$ be a proper projection
with range $\s$. Then $X=VV^+$, where $V=2Q-I$,  is a symmetrizable operator with non real points in the spectrum $\sigma_\e(X)$.
\end{coro} 
\begin{proof}
If the proper subspace $\s$ is not compatible, and $Q$ is a proper projection  with range $\s$, then by Theorem \ref{suff cond} the operator $Q + Q^+ - I$ is not invertible in $\e$. Equivalently, we have that $V + V^+$ is not invertible, where $V=2Q-I$. Since $V^2=I$, we get that $-1 \in \sigma_\e(VV^+)$. 
Now consider the continuous unital monomorphism given by
\[  \vp: \p \to \b(\l), \, \, \, \vp(X)=\bar{X}. \]
Since $\vp$ is a unital morphism, it follows that $\sigma_\l(\bar{X})\subseteq \sigma_\p(X)$ (see also Theorem \ref{proper op}). Moreover, each connected component $\Delta$ of $\sigma_\p(X)$ satisfies $\Delta \cap \sigma_\l (\bar{X}) \neq \emptyset$ (see \cite[Thm. 4.5]{barnes90}). 
If we apply this result to $X=VV^+$, and we take into account that $\sigma_\l (\bar{X})\subseteq (0,\infty)$ and $0 \notin \sigma_\p(X)$, then we find that there is some  $z \in \Delta$ with non trivial imaginary part. Thus, we get 
that $\sigma_\p(X)$ has non real points. Hence $\sigma_\e(X)$ also has non real points.  
\end{proof}

\section{Examples}\label{section examples}

%\subsection{General examples.}

\subsection{Trace class and Hilbert-Schmidt operators.}
In the  examples of this subsection, we take  $\e=(\b_1(\h), \| \, \cdot \, \|_1)$ and $\l=(\b_2(\h), \| \, \cdot \, \|_2)$  the spaces of trace class operators and Hilbert-Schmidt operators on a Hilbert space $\h$, respectively. Recall that $\b_2(\h)$ is a Hilbert space with inner product given by $\PI{x}{y}=Tr(xy^*)$, where $Tr$ denotes the usual trace and $x,y \in \b_2(\h)$. 

\begin{ejem}
 A projection $q$ acting  on $\h$ gives rise to a projection on $\b_1(\h)$ by left multiplication, i.e.
\[   L_q: \b_1(\h) \to \b_1(\h), \, \, \, \, L_q(x)=qx. \] 
We note that 
$ \PI{L_q(x)}{y}=\PI{x}{L_{q^*}(y)}  $
for all $x,y \in \b_1(\h)$. Thus, $L_q$ is proper projection, and $L_q ^+=L_{q^*}$. Then, the range of $L_q$, that is
\[  \s=\{  \, qx \, : \,  x \in \b_1(\h)  \, \}, \]
is a proper subspace. Now we prove that $\s$ is a compatible subspace. Let $\sigma(L_x)$ denote that spectrum of $L_x$ in $\b(\b_1(\h))$ and $\sigma(x)$ denote  the spectrum of $x$ in $\b(\h)$. 
If $\lambda \notin \sigma(x)$, then there exists $y \in \b(\h)$ such that $(x-\lambda)y=y(x-\lambda)=1$. This implies 
$(L_x-\lambda)L_y=L_y(L_x-\lambda)=I$, so that $\sigma(L_x)\subseteq \sigma(x)$. Using this fact with $x=q - q^*$, we see that 
$\sigma(L_q -L_{q^*})\subseteq \sigma(q-q^*)\subseteq i \, \RR$. Also note that 
$(L_q + L_q ^+ - I )^2=I - (L_q - L_q ^+)^2$, then $\sigma((L_q + L_q ^+ - I )^2)=\sigma(I - (L_q -L_q^*)^2) \subseteq [1, \infty]$. 
We conclude that $L_q + L_q ^+ - I$ is invertible on $\b_1(\h)$, and by Theorem \ref{suff cond}, it follows that    $\s$ is a compatible subspace.  
\end{ejem}

\begin{ejem}\label{prop no comp}
Let $q$ be a projection in $\b(\h)$. 
Denote by $C_q$ the following  operator
\[   C_q: \b_1(\h) \to \b_1(\h), \, \, \, \, C_q(x)=qxq.\]
Clearly, $C_q$ is a continuous projection. It is easily seen that $C_q^+=C_{q^*}$, so we have that $C_q$ is a proper projection and 
its range $$\s=\{ \, qxq \, : \, x \in \b_1(\h)\, \}$$ 
is a proper subspace. We shall see below that the compatibility of this subspace depends on our election of the projection $q$. In particular, we shall prove that $\s$ is not compatible for infinitely many different choices of $q$. 

Assume that $R(q)=\k$ is an infinite dimensional subspace of $\h$ satisfying $\k \oplus \k =\h$ (orthogonal sum). We write $q$ as a matrix with respect to this  decomposition of $\h$ as
 $$q=\begin{pmatrix} 1 & z \\ 0 & 0  \end{pmatrix}.$$
Now we prove the following:  
\begin{itemize} \item[i)] If $z \in \b(\k)$, $\|z\|< 1$, then $\s$ is a compatible subspace.
\item[ii)] If $z \in Gl(\k)$, $z$ normal, then $\s$ is a compatible subspace if and only if $\ov{\lambda} \mu \neq - 1$ for all $\lambda, \mu \in \sigma(z)$. In particular, $\s$ is non compatible if $z$ is a self-adjoint symmetry. 
\end{itemize}

\smallskip

\noi $\text{i)}$ We consider the matrix representation of two arbitrary operators $x,y \in \b(\h)$ 
with respect to the  decomposition $\k \oplus \k =\h$:
\[ x=\begin{pmatrix}  x_{11} & x_{12} \\  x_{21} & x_{22} \end{pmatrix}, \, \, \, \, \, \, \, \,
y=\begin{pmatrix}  y_{11} & y_{12} \\  y_{21} & y_{22} \end{pmatrix}.\]
Then,
\begin{equation}\label{rep of s} 
qxq=\begin{pmatrix}  x_{11} + z x_{21} &  (x_{11} + z x_{21})z \\  0 & 0 \end{pmatrix}, 
\end{equation}
and
$$ qxq y^*=\begin{pmatrix}   (x_{11} + z x_{21})y_{11}^*  + (x_{11} + z x_{21})zy_{12}^* &  *       \\   0  &  0      \end{pmatrix}. $$
Thus, $y$ is orthogonal to $\s$ if and only if $Tr((x_{11} + z x_{21})(y_{11}^*  + zy_{12}^*))=0$ for all $x \in \b_1(\h)$. 
Therefore, we obtain
\[  \s^{\perp} \cap \b_1(\h) =\{ \, y \in \b_1(\h)   \, : \,  y_{11} + z^* y_{12}=0 \,  \}. \]
The subspace $\s$ is compatible if and only if $\s \, \dot{+} \, (\s^{\perp} \cap \b_1(\h)) = \b_1(\h)$.
This means that any operator $w \in \b_1(\h)$ can be written as 
\[  w=\begin{pmatrix}  w_{11} & w_{12} \\  w_{21} & w_{22} \end{pmatrix} = \begin{pmatrix}  x_{11} + z x_{21} - z^*y_{12} &  (x_{11} + z x_{21})z + y_{12} \\  y_{21} & y_{22} \end{pmatrix}.  \]
The only non trivial equations to solve are the following
\[  w_{11}=x_{11} + z x_{21} - z^*y_{12}, \, \, \, \, w_{12}=(x_{11}+ zx_{21})z + y_{12} \, .   \]
Put $a=x_{11} + z x_{21}$, $b=y_{12}$, $x=w_{11}$ and $y=w_{12}$. Then $\s$ is a compatible subspace if and only if 
\begin{align*}\label{ec compat}
 x=a - z^*b, \, \, \, \,  \, \, y=az + b \,
\end{align*}
have a solution $a,b \in \b_1(\k)$ for each pair $x,y \in \b_1(\k)$. 
By the first equation, $b=y-az$, so we have to solve $x + z^*y=a + zaz^*=(I+Ad_z)(a)$, where 
$Ad_z:\b_1 (\k)  \to \b_1(\k)$ is defined by $Ad_z(x)=z^*xz$. It is not difficult to see that 
$\|Ad_z\|=\|z\|^2$. Since we suppose $\|z\|<1$, then we have $0 \notin \sigma_{\b_1(\h)}(I +Ad_z )$.
Therefore, $I +Ad_z$ is surjective, and $\s$ is a compatible subspace.

\smallskip

\noi \text{ii)} We now assume that $z$ is an invertible normal operator. In the preceding paragraph, we can 
rewrite the operator $I + Ad_z$ using the left and right multiplication operators on $\b_1(\k)$ given by 
$L_{z^*}(x)=z^*x$ and $R_{-z}(x)=-xz$. Thus, we have 
$  I + Ad_z=L_{z^*}(L_{(z^*)^{-1}} - R_{-z})$. Since $L_{z^*}$ is invertible, the equation $x + z^*y=(I+Ad_z)(a)$ has a solution if and only if
the operator $L_{(z^*)^{-1}} - R_{-z}: \b_1(\k) \to \b_1(\k)$ is surjective. 

Among several equivalent conditions, it was proved in \cite[Thm. 3.2]{fialkow}  that the operator
\[  \tau_1 : \b_1(\k) \to \b_1(\k), \, \, \, \, \tau_1(x)=cx - xd,  \] 
is surjective if and only if $\sigma_r(c)\cap \sigma_l(d)=\emptyset$. Here  $\sigma_r(c)$ is the right spectrum of $c$ and $\sigma_l(d)$ is the left spectrum of $d$. Applying this result to our situation, where $c=(z^*)^{-1}$ and $d=-z$, and using that the right and left spectra of a normal operator coincide with its usual spectrum, we find that  $L_{(z^*)^{-1}} - R_{-z}$ is surjective when 
$\sigma((z^*)^{-1}) \cap \sigma(-z) = \emptyset$. This latter condition can be also written as $\ov{\lambda} \mu \neq - 1$ for all $\lambda, \mu \in \sigma(z)$.
\end{ejem}

\begin{ejem}\label{two diff companions}
From Eq. (\ref{rep of s}) in the previous example, we know that 
$$  \s=\left\{ \, \begin{pmatrix} a & az  \\ 0 & 0  \end{pmatrix} \, : \, a \in \b_1(\k)  \, \right\}. $$
A proper companion of the subspace $\s$ is given by
   
\begin{align*}
\t  =N(C_q)= \left\{ \, \begin{pmatrix} -za & b \\  a & c  \end{pmatrix}  \, : \,a,b,c \in \b_1(\k) \,  \right\}.
\end{align*}
We shall prove the following facts:
\begin{itemize}
\item[i)] The subspace $\t$ is compatible for all $z \in \b(\k)$.
\item[ii)] Let $z \in Gl(\k)$ be  a normal operator such that $\s$ is a compatible subspace. Then there exists an operator $G \in \p^{\times}$ such that $G(\t)=\t$ and $G(\s)$ is a non compatible proper companion of $\t$.  
\end{itemize}

\smallskip

\noi $\text{i)}$ %Note that $\t$ can be rewritten as 
%\[  \t=\left\{ \, \begin{pmatrix} -za & bz \\  a & c  \end{pmatrix}  \, : \,a,b,c \in \b_1(\k) \,  \right\}.  \] 
The orthogonal supplement of $\t$ in $\b_1(\h)$ can be computed:
$$ \t^\perp \cap \b_1(\h)=\left\{ \, \begin{pmatrix} a & 0 \\ z^*a & 0 \end{pmatrix} \, :  \, a \in \b_1(\k) \, \right\}.  $$
Then $\t$ is a compatible subspace if and only if it is possible to solve in $\b_1(\h)$ the following equation: 
\[ \begin{pmatrix} w_{11} & w_{12} \\ w_{21} & w_{22} \end{pmatrix}   = \begin{pmatrix} a_0 & 0 \\  z^*a_0 & 0     \end{pmatrix}  +    \begin{pmatrix}   -za_1 & b_1 \\ a_1 & c_1 \end{pmatrix},  \]
for every $w_{ij} \in \b_1(\k)$, $i,j=1,2$.  The only non trivial equations to solve are 
%\begin{equation}\label{eq 1 comp} 
\[  w_{11}=a_0 - z a_1, \, \, \, \, w_{21}=z^*a_0 + a_1 , \,     \]
%\end{equation}
which always have solutions  given by $a_1=w_{21} -z^*a_0$ and 
$a_0=(1+zz^*)^{-1} (w_{11} + zw_{21})$. 

%Now assume that the system in (\ref{eq 2 comp}) has a solution. This implies 
%that $w_{12}-b_0=b_1 z$ for all $w_{12} \in \b_1(\k)$. Equivalently, $w=b_1z$ has a solution for all $w \in \b_1(\k)$. 
%By Douglas' theorem \cite{douglas}, this means that $R(w^*)\subseteq R(z^*)$, so that $R(z^*)=\h$. Conversely, we now suppose that $z^*$ is a surjective operator. We can take $b_0=0$. Again by Douglas' theorem $w_{12}= b_1 z$ has a solution for every $w_{12} \in  \b_1(\k)$ exactly when $R(w_{12}^*)\subseteq R(z^*)=\h$. Moreover, there is a solution given by 
%$b_1=w_{12}z^{\dagger} \in \b_1(\k)$, where $z^{\dagger}$ is the Moore-Penrose inverse of $z$.    
 
\smallskip

\noi $\text{ii)}$ %From Eq. (\ref{rep of s}), we know that 
%$$  \s=\left\{ \, \begin{pmatrix} a & az  \\ 0 & 0  \end{pmatrix} \, : \, a \in \b_1(\k)  \, \right\}. $$
%If $z$ is invertible, $\t$ can be written as
%$$   \t=\left\{ \, \begin{pmatrix} -za & b \\  a & c  \end{pmatrix}  \, : \,a,b,c \in \b_1(\k) \,  \right\}.  $$  
Suppose that $x \in Gl(\h)$ has the matrix representation
$$  x=\begin{pmatrix}  z & 0  \\ 0 & t \end{pmatrix}, $$
where $t$ is a self-adjoint symmetry on $\k$. We take $G=R_x: \b_1 (\h) \to \b_1(\h)$, $G(y)=yx$. Clearly, $G$ is a proper operator, $G^+=R_{x^*}$, and both $G$ and $G^+$ are invertible on $\b_1(\h)$. Thus, we get $G \in \p^{\times}$. 

We observe that 
$$  \begin{pmatrix} -za & b \\  a & c  \end{pmatrix}\begin{pmatrix}  z & 0  \\ 0 & t \end{pmatrix}=\begin{pmatrix}  -zaz &  bt  \\ az & ct  \end{pmatrix}, $$
hence, we get $G(\t)\subseteq \t$. Similarly, $G^{-1}(\t)\subseteq \t$, so we obtain $G(\t)=\t$. 
According to Proposition \ref{relacion supl propers}, the subspace
\[ G(\s)=\left\{  \, \begin{pmatrix}  b & bt  \\ 0 & 0 \end{pmatrix}  \, : \, b \in \b_1(\k) \,  \right\} \]
is a proper companion of $\t$. By the item $\text{ii)}$ of Example \ref{prop no comp} we get that $G(\s)$ is not compatible.    
\end{ejem}

\subsection{Proper invertible operators.}

Proper operators have three different notions of inverses (see Remark \ref{notions of spec}). In this subsection we study
proper invertible operators. 

\begin{ejem}
Invertible operators in $\e$ which are isometric for the $\l$ inner product. We shall call them \textit{unitarizable operators}. In the special case when $\e=\h$ is a Hilbert space, these were studied in \cite{achl} and \cite{anchdi}. They can be obtained, for instance, as exponentials $A=e^{iX}$, with $X$ a symmetrizable operator. But not every $\l$-isometric operator is an exponential (see \cite[Example 4.9]{achl}).
\end{ejem}

\begin{ejem}
A special case of the above example occurs if we take $\e=\b_1(\h)$ and $\l=\b_2(\h)$. Let $u,v$ be unitary operators in $\h$, and denote by $x^t$ the transpose of $x\in\b(\h)$ with respect to a given orthonormal basis of $\h$. Then the operators
$$
\mu_{u,v}, \theta_{u,v}\in\b(\b_1(\h)), \ \ \mu_{u,v}(x)=uxv \hbox{ and } \theta_{u,v}(x)=ux^tv
$$
are isometric  for the norms $\|\cdot \|_p$ for any $1\le p \le \infty$ (for $p\ne 2$, any isometry for the $p$ norm is of this type \cite{arazy}). Thus $\mu_{u,v}$ and $\theta_{u,v}$ are invertible in $\p$ (in fact they are unitarizable). If one replaces the unitaries $u,v$ by invertible operators $g, h$ in $\h$, then $\mu_{g,h}$ and $\theta_{g,h}$ are proper and invertible operators in $\b_1(\h)$ with inverses $\mu_{g^{-1},h^{-1}}$ and $\theta_{g^{-1},h^{-1}}$ which are also proper operators. 
\end{ejem}

\begin{ejem}
Let $\e=H^1_0(\Omega)$ be the Sobolev space of the domain $\Omega\subset\mathbb{R}^n$, i.e. the completion of $C_0^\infty(\Omega)$ with the inner product norm
$$
\|f\|_{1}=\int_\Omega |f(x)|^2 + |(\nabla f)(x)|^2 dx.
$$
Let $\l=L^2(\Omega , dx)$ be the Lebesgue space of square-integrable functions endowed with its usual inner product. Pick a function $\varphi$, which is $C^1$ in $\Omega$, continuous and bounded in $\bar{\Omega}$ and satisfies $|\varphi(x)|\ge r>0$ for $x\in\Omega$. Then the multiplication operator
$$
M_\varphi f=\varphi f,
$$
preserves $\e$. Its adjoint in $\l$, which is $M_{\bar{\varphi}}$, also preserves $\e$. Thus $M_\varphi$ is a proper operator. Its inverse $M_{\frac{1}{\varphi}}$ also belongs to $\p$. Thus, $M_\varphi\in \p^{\times}$. Moreover, apparently
$$
\sigma_\l(M_\varphi)=\sigma_E(M_\varphi)=\sigma_\p(M_\varphi)=\varphi(\bar{\Omega}).
$$
\end{ejem}
There is another situation in which the three spectra coincide. 

\begin{prop}
Let $G\in\p$ such that $G-I$ and $G^+ -I$ are compact operators on $\e$. Assume that its extension $\bar{G}$ is invertible in $\l$, then $G \in \p^{\times}$. 
\end{prop}
\begin{proof}
The set of invertible operators $G$  in $\e$ such that $G-I$ is compact form a closed subgroup of the invertible group of $\e$ (it is sometimes called the Fredholm group of $\e$). Let $G=I+K$ for some $K$ compact in $\e$. The operator $K$ is proper, and therefore   $\bar{K}$ is compact in $\l$, with the same (non nil) eigenvalues as $K$. Furthermore, the multiplicity of each nonzero eigenvalue is the same over $\e$  and $\l$  (see Theorem \ref{proper op}). Thus $0$ does not belong to the spectrum of $G$. Since $K^+=\bar{K}^*|_\e$ is also compact on $\e$, $\bar{K}^*$ is compact on $\l$, and its eigenvalues are the conjugates of the eigenvalues of $K$. It follows that  $G^+=I+K^+$ has trivial kernel, and thus, by the Fredholm alternative, it is invertible in $\e$.
\end{proof}

\begin{rem}
Unitarizable operators preserve compatible subspaces: if $G$ is unitarizable and  $\s$ is compatible, then $G(\s)$ is also compatible. This allows to produce more examples of proper subspaces which are not compatible. Namely, if $\s$ is a proper  non compatible subspace and $G$ is unitarizable, then $G(\s)$ is a proper  non compatible subspace.
\end{rem}

\begin{ejem}
Consider again Example \ref{prop no comp} for some projection $q$ such that $\s$ is non compatible. Let $u,v$ be unitary operators in $\h$. Then by the above remark, if
$$
\s=\{qxq: x\in\b_1(\h)\},
$$
then $\mu_{u,v}(\s)$ is proper but non compatible. Explicitly, this subspace is the range of the idempotent $\mu_{u,v} C_q \mu_{u^*,v^*}$. Then
$$
\mu_{u,v}(\s)=\{(uqu^*)x(vqv^*): x\in\b_1(\h)\}.
$$
Thus the subspaces $\s_{q_1,q_2}=\{q_1xq_2: x \in \b_1(\h)\}$ are proper and non compatible, if $q_i$ are chosen in the unitary orbit of $q$.
\end{ejem}

\section*{Acknowledgments}

We thank Gustavo Corach for several helpful conversations on compatible subspaces and Alejandra Maestripieri for providing useful references.  
All the authors are partially supported by PIP CONICET 0757. The second author is also partially supported by UNLP 11X585.

\bigskip
\bigskip

\noi {\sc (Esteban Andruchow)} {Instituto de Ciencias,  Universidad Nacional de Gral. Sar\-miento,
J.M. Gutierrez 1150,  (1613) Los Polvorines, Argentina and Instituto Argentino de Matem\'atica, `Alberto P. Calder\'on', CONICET, Saavedra 15 3er. piso,
(1083) Buenos Aires, Argentina.}

\noi e-mail: {\sf eandruch@ungs.edu.ar}

\bigskip

\noi {\sc (Eduardo Chiumiento)} {Departamento de de Matem\'atica, FCE-UNLP,Calles 50 y 115, 
(1900) La Plata, Argentina  and Instituto Argentino de Matem\'atica, `Alberto P. Calder\'on', CONICET, Saavedra 15 3er. piso,
(1083) Buenos Aires, Argentina.}     
                                               
\noi e-mail: {\sf eduardo@mate.unlp.edu.ar}

\bigskip

\noi {\sc (Mar\'ia Eugenia Di Iorio y Lucero)} {Instituto de Ciencias,  Universidad Nacional de Gral. Sarmiento,
J.M. Gutierrez 1150,  (1613) Los Polvorines, Argentina and Instituto Argentino de Matem\'atica, `Alberto P. Calder\'on', CONICET, Saavedra 15 3er. piso,
(1083) Buenos Aires, Argentina.}

\noi e-mail: {\sf mdiiorio@ungs.edu.ar }

\end{document}